\documentclass[a4paper]{article}

\usepackage{geometry}
\usepackage{amsmath,amssymb,amsfonts,amsthm}
\usepackage{indentfirst}
\usepackage{graphicx}
\usepackage[colorlinks=true]{hyperref}
\usepackage{mathrsfs}
\usepackage{chngcntr}

\numberwithin{equation}{section}

\newtheorem{Thm}{Theorem}[section]
\newtheorem{Cor}[Thm]{Corollary}
\newtheorem{Lem}[Thm]{Lemma}

\newtheorem{Rek}[Thm]{Remark}

\newcommand{\N}{\mathbb{N}}

\renewcommand{\i}{\boldsymbol{i}\mkern1mu}
\renewcommand{\j}{\boldsymbol{j}\mkern1mu}
\renewcommand{\k}{\boldsymbol{k}\mkern1mu}

\title{Improved Sobolev Inequalities on the Quaternionic Sphere}

\author{Zongxiong Ren, Zhipeng Yang\thanks{Corresponding author: Z.~Yang.}}

\date{}

\AtEndDocument{%
  \par
  \bigskip
  \bigskip

  \noindent
  \textbf{Zongxiong Ren:}\\[0.2em]
  \textsc{Department of Mathematics, Yunnan Normal University, Kunming, China}\\[0.3em]
  \textit{E-mail address}: \texttt{2448783498@qq.com}\\[1.5em]

  \noindent
  \textbf{Zhipeng Yang:}\\[0.2em]
  \textsc{Department of Mathematics, Yunnan Normal University, Kunming, China}\\
  \textsc{Yunnan Key Laboratory of Modern Analytical Mathematics and Applications, Kunming, China}\\[0.3em]
  \textit{E-mail address}: \texttt{yangzhipeng326@163.com}%
}

\begin{document}
\maketitle

\begin{abstract}
In this paper we establish improved Sobolev inequalities on the quaternionic sphere under higher-order moment vanishing conditions with respect to the measure \(|u|^{p^*}\,d\xi\). As an application, we give a new proof of the existence of extremals for the sharp Sobolev embedding
\[
S^{1,2}(S^{4n+3}) \hookrightarrow L^{2^*}(S^{4n+3}).
\]


\end{abstract}

\section{Introduction and main results}

The main goal of this paper is to establish improved Sobolev inequalities on the quaternionic sphere under higher-order moment vanishing conditions with respect to the measure \(|u|^{p^*}\,d\xi\). The precise statement is given in Theorem~\ref{thm:improved-jk}. As an application, we also give a new proof of the existence of extremals for the sharp Sobolev embedding
\[
S^{1,2}(S^{4n+3}) \hookrightarrow L^{2^*}(S^{4n+3}).
\]

Our proof of Theorem~\ref{thm:improved-jk} is inspired by \cite{MR4373170,MR4452442}, where Aubin's Moser--Trudinger--Onofri inequality and Sobolev inequalities on the standard sphere \(S^n\) were refined under higher-order moment vanishing assumptions. The basic strategy in \cite{MR4452442} is the following. If the improved inequality failed, one could construct a sequence \(\{F_i\}\subset W^{1,p}(S^n)\) such that
\[
\|\nabla F_i\|_{L^p(S^n)}^p \le \frac{1}{\alpha},
\qquad
\|F_i\|_{L^{p^*}(S^n)} = 1,
\]
and \(F_i \rightharpoonup 0\) weakly in \(L^p(S^n)\). Here \(p^*=\frac{np}{n-p}\), and \(\alpha\) denotes the leading coefficient in the corresponding improved Sobolev inequality. Applying the concentration--compactness principle from \cite{MR850686} then leads to a contradiction. We adapt this argument to the quaternionic sphere. A key ingredient is the concentration--compactness principle in the quaternionic contact setting, stated later as Lemma~\ref{lem:cc}.

Let \(G\) be a Carnot group of homogeneous dimension \(Q\), and let \(\Omega\subset G\) be open. In the quaternionic Heisenberg case one has
\[
G=\boldsymbol{G}=\mathbb{H}^{n}\times \operatorname{Im}\mathbb{H},
\qquad
Q=4n+6.
\]
Ivanov, Minchev, and Vassilev proved in \cite{2010The} that for every \(1<p<Q\) there exists a constant \(C_p^{QC}>0\) such that
\begin{equation}\label{eq:FS-group}
\left(\int_{\Omega}|u|^{p^*}\,dH\right)^{1/p^*}
\le
C_p^{QC}
\left(\int_{\Omega}|Xu|^p\,dH\right)^{1/p},
\qquad
u\in C_0^\infty(\Omega),
\end{equation}
where
\[
p^*=\frac{Qp}{Q-p},
\qquad
|Xu|=\left(\sum_{j=1}^m |X_j u|^2\right)^{1/2},
\]
and \(X_1,\dots,X_m\) form a basis of the first layer of \(G\).

Via the quaternionic Cayley transform, inequality \eqref{eq:FS-group} is equivalent to a Sobolev inequality on the quaternionic sphere. Throughout the paper, we denote by \(S_p^{QC}\) the best constant in the sphere form
\begin{equation}\label{eq:FS-sphere}
\left(\int_{S^{4n+3}} |u|^{p^*}\,d\xi\right)^{p/p^*}
\le
S_p^{QC}
\int_{S^{4n+3}} |\nabla_{qc}u|^p\,d\xi,
\qquad
u\in S^{1,p}(S^{4n+3}),
\end{equation}
where \(d\xi\) is the standard surface measure on \(S^{4n+3}\), and \(\nabla_{qc}u\) denotes the horizontal gradient of \(u\) on the quaternionic sphere.

For \(p=2\), the sharp constant and the extremals for the corresponding Folland--Stein inequality on the quaternionic Heisenberg group \(\boldsymbol{G}\) were computed in \cite{2010The}. More precisely, one has
\begin{equation}\label{eq:sharp-constant}
S_2^{QC}
=
\frac{\left[2^{-2n}\omega_{4n+3}\right]^{-1/(4n+6)}}{2\sqrt{n(n+1)}},
\end{equation}
where
\[
\omega_{4n+3}=\frac{2\pi^{2n+2}}{(2n+1)!}
\]
is the volume of the unit sphere \(S^{4n+3}\subset \mathbb{R}^{4n+4}\). Equality in \eqref{eq:FS-group} is attained precisely by the functions of the form
\begin{equation}\label{eq:extremal-group}
F(q,\omega)
=
\gamma \bigl[(1+|q|^2)^2+|\omega|^2\bigr]^{-(n+1)},
\qquad
\gamma \in \mathbb{R},
\end{equation}
together with all their translations and dilations.

These constants agree with those in \cite{MR2654087,MR1813232}, after taking into account the normalization conventions and the standard identities
\[
\Gamma(n+1)=n!,
\qquad
\Gamma(z+n)=z(z+1)\cdots (z+n-1)\Gamma(z),
\qquad
n\in\N,
\]
\[
\omega_n=\frac{2\pi^{n/2}}{\Gamma(n/2)},
\]
and
\[
\Gamma(2z)=2^{2z-1}\pi^{-1/2}\Gamma(z)\Gamma\!\left(z+\frac{1}{2}\right).
\]

\begin{Rek}
With the left-invariant basis used in \cite[Theorem 1.1]{2010The}, the quaternionic Heisenberg group is not a group of Heisenberg type. If one instead works in the Heisenberg-type normalization, then the best constant in the \(L^2\) Folland--Stein inequality is, by \cite[Theorem 1.6]{MR1813232},
\[
S_2
=
\frac{1}{\sqrt{4n(4n+4)}}\,
4^{3/(4n+6)}
\pi^{-(4n+3)/(2(4n+6))}
\left(\frac{\Gamma(4n+3)}{\Gamma((4n+3)/2)}\right)^{1/(4n+6)},
\]
and the extremals are given by translations and dilations of
\[
F(q,\omega)
=
\gamma \bigl[(1+|q|^2)^2+16|\omega|^2\bigr]^{-(n+1)},
\qquad
(q,\omega)\in G.
\]
\end{Rek}

We identify \(\mathbb{H}^{n+1}\) with \(\mathbb{C}^{2n+2}\). Recall that a homogeneous polynomial on \(\mathbb{C}^{2n+2}\) with bidegree \((j,k)\) is a polynomial \(g\) such that
\[
g(\lambda z)=\lambda^j \overline{\lambda}^{\,k} g(z),
\qquad
\lambda\in\mathbb{C},\ z\in\mathbb{C}^{2n+2}.
\]
For nonnegative integers \(j\) and \(k\), define
\[
\mathcal{P}_{j,k}
=
\left\{
g|_{S^{4n+3}}:
g \text{ is a homogeneous polynomial on }\mathbb{C}^{2n+2}
\text{ with bidegree }(j,k)
\right\},
\]
\[
\widetilde{\mathcal{P}}_{j,k}
=
\sum_{\substack{0\le \tilde j\le j\\ 0\le \tilde k\le k}}
\mathcal{P}_{\tilde j,\tilde k},
\]
and
\[
\overline{\mathcal{P}}_{j,k}
=
\left\{
g\in \widetilde{\mathcal{P}}_{j,k}:
\int_{S^{4n+3}} g\,d\xi =0
\right\}.
\]

For a nonnegative integer \(\ell\), define
\[
\mathcal{P}_{\ell}
=
\left\{
g|_{S^{4n+3}}:
g \text{ is a real polynomial on }\mathbb{R}^{4n+4}
\text{ of degree at most } \ell
\right\},
\]
and
\[
\overline{\mathcal{P}}_{\ell}
=
\left\{
g\in \mathcal{P}_{\ell}:
\int_{S^{4n+3}} g\,d\xi =0
\right\}.
\]

For \(0\le \theta\le 1\), define
\[
\mathcal{M}_{j,k}^{c}(S^{4n+3})
=
\left\{
\nu:
\begin{array}{l}
\nu \text{ is a probability measure on } S^{4n+3},\\
\nu \text{ is supported on countably many points, and}\\
\int_{S^{4n+3}} g\,d\nu =0 \text{ for all } g\in \overline{\mathcal{P}}_{j,k}
\end{array}
\right\},
\]
and
\[
\Theta(j,k;\theta,4n+3)
=
\inf
\left\{
\sum_i \nu_i^\theta:
\nu=\sum_i \nu_i\delta_{x_i}\in \mathcal{M}_{j,k}^{c}(S^{4n+3})
\right\}.
\]

Similarly, define
\[
\mathcal{M}_{\ell}^{c}(S^{4n+3})
=
\left\{
\nu:
\begin{array}{l}
\nu \text{ is a probability measure on } S^{4n+3},\\
\nu \text{ is supported on countably many points, and}\\
\int_{S^{4n+3}} g\,d\nu =0 \text{ for all } g\in \overline{\mathcal{P}}_{\ell}
\end{array}
\right\},
\]
and
\[
\Theta(\ell;\theta,4n+3)
=
\inf
\left\{
\sum_i \nu_i^\theta:
\nu=\sum_i \nu_i\delta_{x_i}\in \mathcal{M}_{\ell}^{c}(S^{4n+3})
\right\}.
\]

Our first main result gives an improved Sobolev inequality under higher-order moment vanishing conditions.

\begin{Thm}\label{thm:improved-jk}
Let \(S^{4n+3}\) be the quaternionic unit sphere. Let \(Q=4n+6\), let \(1<p<Q\), and let \(p^*=\frac{Qp}{Q-p}\). Then for every \(\varepsilon>0\) there exists a constant \(C(\varepsilon)>0\) such that
\[
\left(\int_{S^{4n+3}} |u|^{p^*}\,d\xi\right)^{p/p^*}
\le
\left(
\frac{S_p^{QC}}{\Theta\!\left(j,k;\frac{Q-p}{Q},4n+3\right)}
+\varepsilon
\right)
\int_{S^{4n+3}} |\nabla_{qc}u|^p\,d\xi
+
C(\varepsilon)\int_{S^{4n+3}} |u|^p\,d\xi
\]
for every \(u\in S^{1,p}(S^{4n+3})\) satisfying
\begin{equation}\label{eq:moment-jk}
\int_{S^{4n+3}} g\,|u|^{p^*}\,d\xi=0
\qquad
\text{for all } g\in \overline{\mathcal{P}}_{j,k}.
\end{equation}
\end{Thm}

Following \cite{MR4452442}, one may also work with the spaces \(\overline{\mathcal{P}}_{\ell}\). This yields the following variant.

\begin{Thm}\label{thm:improved-l}
Let \(S^{4n+3}\) be the quaternionic unit sphere. Let \(Q=4n+6\), let \(1<p<Q\), and let \(p^*=\frac{Qp}{Q-p}\). Then for every \(\varepsilon>0\) there exists a constant \(C(\varepsilon)>0\) such that
\[
\left(\int_{S^{4n+3}} |u|^{p^*}\,d\xi\right)^{p/p^*}
\le
\left(
\frac{S_p^{QC}}{\Theta\!\left(\ell;\frac{Q-p}{Q},4n+3\right)}
+\varepsilon
\right)
\int_{S^{4n+3}} |\nabla_{qc}u|^p\,d\xi
+
C(\varepsilon)\int_{S^{4n+3}} |u|^p\,d\xi
\]
for every \(u\in S^{1,p}(S^{4n+3})\) satisfying
\begin{equation}\label{eq:moment-l}
\int_{S^{4n+3}} g\,|u|^{p^*}\,d\xi=0
\qquad
\text{for all } g\in \overline{\mathcal{P}}_{\ell}.
\end{equation}
\end{Thm}

As an immediate consequence of Theorem~\ref{thm:improved-l} with \(\ell=1\), we obtain the balanced first-moment case.

\begin{Cor}\label{cor:first-moment}
Let \(S^{4n+3}\) be the quaternionic unit sphere. Let \(Q=4n+6\), let \(1<p<Q\), and let \(p^*=\frac{Qp}{Q-p}\). Then for every \(\varepsilon>0\) there exists a constant \(C(\varepsilon)>0\) such that
\[
\left(\int_{S^{4n+3}} |u|^{p^*}\,d\xi\right)^{p/p^*}
\le
\left(
\frac{S_p^{QC}}{\Theta\!\left(1;\frac{Q-p}{Q},4n+3\right)}
+\varepsilon
\right)
\int_{S^{4n+3}} |\nabla_{qc}u|^p\,d\xi
+
C(\varepsilon)\int_{S^{4n+3}} |u|^p\,d\xi
\]
for every \(u\in S^{1,p}(S^{4n+3})\) satisfying
\[
\int_{S^{4n+3}} \xi_i |u|^{p^*}\,d\xi=0,
\qquad
i=1,\dots,4n+4.
\]
\end{Cor}

As an application, we give a new proof of the existence of extremals for the sharp Sobolev inequality
\begin{equation}\label{eq:sharp-sobolev}
\left(\int_{S^{4n+3}} |u|^{2^*}\,d\xi\right)^{2/2^*}
\le
S_2^{QC}\int_{S^{4n+3}} |\nabla_{qc}u|^2\,d\xi.
\end{equation}

\begin{Thm}\label{thm:extremal}
Let \(S^{4n+3}\) be the quaternionic unit sphere, and let
\[
2^*=\frac{2Q}{Q-2},
\qquad
Q=4n+6.
\]
Then the best constant in \eqref{eq:sharp-sobolev} is attained. Moreover, if \(\{u_i\}\) is a minimizing sequence, then there exist a subsequence \(\{u_{i_m}\}\) and a sequence \(\{\Phi_m\}\subset \mathrm{Aut}_{QC}(S^{4n+3})\) of quaternionic contact conformal automorphisms such that
\[
u_{i_m}^{\Phi_m}
=
|J_{\Phi_m}|^{1/2^*}\,\Phi_m^*u_{i_m}
\]
converges strongly in \(S^{1,2}(S^{4n+3})\), where \(J_{\Phi_m}\) denotes the Jacobian determinant of \(\Phi_m\) with respect to the surface measure \(d\xi\).
\end{Thm}

\section{Preliminaries}

In this section, we collect some standard facts which will be used later. Our main reference is \cite{MR4074054}.

Let
\[
\mathbb{H}^{n+1}
=
\left\{
(\zeta_1,\dots,\zeta_{n+1}) : \zeta_\ell=a_\ell+b_\ell \i+c_\ell \j+d_\ell \k,\ a_\ell,b_\ell,c_\ell,d_\ell\in\mathbb{R}
\right\}
\]
be the quaternionic vector space of dimension \(n+1\). The quaternionic units satisfy
\[
\i^2=\j^2=\k^2=-1,
\qquad
\i \j=\k,\quad \j \k=\i,\quad \k \i=\j,
\]
and hence
\[
\j\i=-\k,\qquad \k\j=-\i,\qquad \i\k=-\j.
\]

For \(\zeta\in \mathbb{H}^{n+1}\), we write
\[
\zeta=a+b\i+c\j+d\k,
\]
where \(a,b,c,d\in \mathbb{R}^{n+1}\). Fixing the complex structure determined by left multiplication by \(\i\), we may also write
\[
\zeta=(a+b\i)+(c+d\i)\j=z+w\,\j,
\qquad
z=a+b\i,\quad w=c+d\i,
\]
so that \(\mathbb{H}^{n+1}\simeq \mathbb{C}^{2n+2}\).

We endow \(\mathbb{H}^{n+1}\) with the standard quaternionic Hermitian product
\[
\langle \zeta,\eta\rangle=\sum_{\ell=1}^{n+1}\zeta_\ell \overline{\eta_\ell}.
\]
Under the identification \(\mathbb{H}^{n+1}\simeq \mathbb{R}^{4n+4}\), the unit sphere is
\[
S^{4n+3}
=
\left\{
\zeta\in \mathbb{H}^{n+1} : \langle \zeta,\zeta\rangle=1
\right\}.
\]
Its tangent space and horizontal space at \(\zeta\in S^{4n+3}\) are given by
\[
T_\zeta S^{4n+3}
=
\left\{
\eta\in \mathbb{H}^{n+1} : \Re \langle \zeta,\eta\rangle=0
\right\},
\]
and
\[
H_\zeta S^{4n+3}
=
\left\{
\eta\in \mathbb{H}^{n+1} : \langle \zeta,\eta\rangle=0
\right\}.
\]

Consider the vector fields induced by left multiplication,
\[
T_i(\zeta)=-\i\zeta,\qquad T_j(\zeta)=-\j\zeta,\qquad T_k(\zeta)=-\k\zeta.
\]
For example,
\[
-\i\zeta
=
-\i(a+b\i+c\j+d\k)
=
b-a\i+d\j-c\k,
\]
and therefore
\[
\Re\langle \zeta,-\i\zeta\rangle=0.
\]
Similarly,
\[
\Re\langle \zeta,-\j\zeta\rangle=0,
\qquad
\Re\langle \zeta,-\k\zeta\rangle=0.
\]
Hence \(T_i,T_j,T_k\) are tangent to \(S^{4n+3}\). They span the vertical distribution, and
\[
T_\zeta S^{4n+3}
=
H_\zeta S^{4n+3}\oplus \mathrm{span}\{T_i(\zeta),T_j(\zeta),T_k(\zeta)\}.
\]

The groups \(Sp(n+1)\) and \(\overline{Sp(1)}\) act by orthogonal transformations preserving the horizontal distribution. In particular,
\[
\overline{Sp(1)}\cdot Sp(n+1)
=
\overline{Sp(1)}\times Sp(n+1)\big/\mathbb{Z}_2
\]
is the compact subgroup of quaternionic contact isometries of \(S^{4n+3}\). We denote by
\[
\mathrm{Aut}_{QC}(S^{4n+3})
\]
the full quaternionic contact conformal automorphism group of \(S^{4n+3}\), which is larger than \(\overline{Sp(1)}\cdot Sp(n+1)\).

For a smooth function \(f\) on \(S^{4n+3}\), the vector field \(T_i\) acts as
\[
(T_i f)(\zeta)
=
\frac{d}{dt}\bigg|_{t=0} f(e^{-t\i}\zeta).
\]
Since
\[
e^{-t\i}\zeta
=
(\cos t-\i\sin t)(a+b\i+c\j+d\k)
\]
\[
=
(a\cos t+b\sin t)
+
(b\cos t-a\sin t)\i
+
(c\cos t+d\sin t)\j
+
(d\cos t-c\sin t)\k,
\]
we obtain
\[
(T_i f)(\zeta)
=
\sum_{\ell=1}^{n+1}
\left(
b_\ell \frac{\partial f}{\partial a_\ell}
-
a_\ell \frac{\partial f}{\partial b_\ell}
+
d_\ell \frac{\partial f}{\partial c_\ell}
-
c_\ell \frac{\partial f}{\partial d_\ell}
\right).
\]
In the same way,
\[
(T_j f)(\zeta)
=
\sum_{\ell=1}^{n+1}
\left(
c_\ell \frac{\partial f}{\partial a_\ell}
-
d_\ell \frac{\partial f}{\partial b_\ell}
-
a_\ell \frac{\partial f}{\partial c_\ell}
+
b_\ell \frac{\partial f}{\partial d_\ell}
\right),
\]
and
\[
(T_k f)(\zeta)
=
\sum_{\ell=1}^{n+1}
\left(
d_\ell \frac{\partial f}{\partial a_\ell}
+
c_\ell \frac{\partial f}{\partial b_\ell}
-
b_\ell \frac{\partial f}{\partial c_\ell}
-
a_\ell \frac{\partial f}{\partial d_\ell}
\right).
\]

Let \(\Delta_{\mathbb{R}^{d+1}}\) and \(\Delta_{S^d}\) denote the Laplacians on \(\mathbb{R}^{d+1}\) and \(S^d\), respectively. For any \(f\in C^\infty(\mathbb{R}^{d+1}\setminus\{0\})\), the expression in spherical coordinates is
\[
\Delta_{\mathbb{R}^{d+1}} f
=
\frac{\partial^2 f}{\partial r^2}
+
\frac{d}{r}\frac{\partial f}{\partial r}
+
\frac{1}{r^2}\Delta_{S^d}f.
\]
Equivalently,
\[
r^2\Delta_{\mathbb{R}^{d+1}} f
=
\left(r\frac{\partial}{\partial r}\right)^2 f
+
(d-1)\left(r\frac{\partial}{\partial r}\right)f
+
\Delta_{S^d}f.
\]

Now assume that \(f\) is a harmonic polynomial homogeneous of degree \(k\). Since
\[
r\frac{\partial}{\partial r}
=
\sum_i x^i \frac{\partial}{\partial x^i}
\]
is the Euler vector field, one has
\[
r\frac{\partial f}{\partial r}=kf.
\]
Hence
\[
0
=
r^2\Delta_{\mathbb{R}^{d+1}}f
=
k^2f+(d-1)kf+\Delta_{S^d}f,
\]
and therefore
\[
\Delta_{S^d}f=-k(k+d-1)f.
\]
Thus every \(k\)-homogeneous harmonic polynomial is an eigenfunction of \(\Delta_{S^d}\) with eigenvalue \(-k(k+d-1)\).

\section{Improved Sobolev Inequalities on the Quaternionic Sphere}

In this section we prove Theorem~\ref{thm:improved-jk}. We also record the analogous statement for polynomial constraints of degree at most \(\ell\), as well as the balanced first-moment case.

We continue to use the notation introduced in Section~1. In particular,
\[
\mathcal M_{j,k}^{c}(S^{4n+3}),\qquad
\mathcal M_{\ell}^{c}(S^{4n+3}),\qquad
\Theta(j,k;\theta,4n+3),\qquad
\Theta(\ell;\theta,4n+3)
\]
are defined as in the introduction.

The proof of Theorem~\ref{thm:improved-jk} relies on the following concentration--compactness lemma on the quaternionic sphere, which is a straightforward adaptation of Lemma A.0.1 in \cite{MR4694907}. We omit the proof.

\begin{Lem}\label{lem:cc}
Let \(\{u_m\}\) be a bounded sequence in \(S^{1,p}(S^{4n+3})\), where \(Q=4n+6\), \(1<p<Q\), and \(p^*=\frac{Qp}{Q-p}\). Assume that
\[
u_m \rightharpoonup u
\qquad
\text{weakly in } S^{1,p}(S^{4n+3}),
\]
and that, in the sense of measures,
\[
|u_m|^{p^*}\,d\xi \rightharpoonup |u|^{p^*}\,d\xi+\nu,
\qquad
|\nabla_{qc}u_m|^p\,d\xi \rightharpoonup |\nabla_{qc}u|^p\,d\xi+\sigma.
\]
Then there exist at most countably many points \(\{x_i\}\subset S^{4n+3}\) and nonnegative numbers \(\{\nu_i\}\), \(\{\sigma_i\}\) such that
\[
\nu=\sum_i \nu_i \delta_{x_i},
\qquad
\sigma \ge \sum_i \sigma_i \delta_{x_i},
\]
and
\[
\nu_i^{p/p^*}\le S_p^{QC}\sigma_i
\qquad
\text{for every } i.
\]
\end{Lem}

We now give the proof of Theorem~\ref{thm:improved-jk}.

\begin{proof}[Proof of Theorem~\ref{thm:improved-jk}]
Set
\[
\theta=\frac{Q-p}{Q}=\frac{p}{p^*},
\qquad
\alpha
=
\frac{S_p^{QC}}{\Theta(j,k;\theta,4n+3)}+\varepsilon.
\]
Assume by contradiction that the conclusion is false. Then for every \(m\in \N\) there exists \(u_m\in S^{1,p}(S^{4n+3})\) such that
\begin{equation}\label{eq:jk-moment-constraint}
\int_{S^{4n+3}} g\,|u_m|^{p^*}\,d\xi=0
\qquad
\text{for all } g\in \overline{\mathcal P}_{j,k},
\end{equation}
and
\[
\left(\int_{S^{4n+3}} |u_m|^{p^*}\,d\xi\right)^{p/p^*}
>
\alpha \int_{S^{4n+3}} |\nabla_{qc}u_m|^p\,d\xi
+
m\int_{S^{4n+3}} |u_m|^p\,d\xi.
\]

By scaling, we may assume
\[
\int_{S^{4n+3}} |u_m|^{p^*}\,d\xi=1.
\]
Hence
\begin{equation}\label{eq:uniform-bounds}
\int_{S^{4n+3}} |\nabla_{qc}u_m|^p\,d\xi<\frac{1}{\alpha},
\qquad
\int_{S^{4n+3}} |u_m|^p\,d\xi<\frac{1}{m}.
\end{equation}
It follows that \(\{u_m\}\) is bounded in \(S^{1,p}(S^{4n+3})\), while \(u_m\to 0\) in \(L^p(S^{4n+3})\). Passing to a subsequence if necessary, we may assume
\[
u_m \rightharpoonup 0
\qquad
\text{weakly in } S^{1,p}(S^{4n+3}),
\]
and, in the sense of measures,
\begin{equation}\label{eq:measure-limits}
|u_m|^{p^*}\,d\xi \rightharpoonup \nu,
\qquad
|\nabla_{qc}u_m|^p\,d\xi \rightharpoonup \sigma.
\end{equation}

By Lemma~\ref{lem:cc}, there exist at most countably many points \(\{x_i\}\subset S^{4n+3}\) such that
\begin{equation}\label{eq:atomic-decomposition}
\nu=\sum_i \nu_i \delta_{x_i},
\qquad
\sigma \ge \sum_i \sigma_i \delta_{x_i},
\qquad
\nu_i^{\theta}\le S_p^{QC}\sigma_i.
\end{equation}
Moreover, from the normalization and \eqref{eq:uniform-bounds} we have
\begin{equation}\label{eq:mass-bounds}
\nu(S^{4n+3})=1,
\qquad
\sigma(S^{4n+3})\le \frac{1}{\alpha}.
\end{equation}

Now let \(g\in \overline{\mathcal P}_{j,k}\). Since \(g\) is continuous on the compact sphere, passing to the limit in \eqref{eq:jk-moment-constraint} gives
\[
\int_{S^{4n+3}} g\,d\nu
=
\lim_{m\to\infty}
\int_{S^{4n+3}} g\,|u_m|^{p^*}\,d\xi
=
0.
\]
Hence \(\nu\in \mathcal M_{j,k}^{c}(S^{4n+3})\). By the definition of \(\Theta(j,k;\theta,4n+3)\), together with \eqref{eq:atomic-decomposition} and \eqref{eq:mass-bounds}, we obtain
\[
\Theta(j,k;\theta,4n+3)
\le
\sum_i \nu_i^{\theta}
\le
S_p^{QC}\sum_i \sigma_i
\le
S_p^{QC}\sigma(S^{4n+3})
\le
\frac{S_p^{QC}}{\alpha}.
\]
Therefore
\[
\alpha \le \frac{S_p^{QC}}{\Theta(j,k;\theta,4n+3)},
\]
which contradicts the choice of
\[
\alpha=\frac{S_p^{QC}}{\Theta(j,k;\theta,4n+3)}+\varepsilon.
\]
The proof is complete.
\end{proof}

\begin{proof}[Proof of Theorem~\ref{thm:improved-l}]
The proof is identical to that of Theorem~\ref{thm:improved-jk}, with \(\overline{\mathcal P}_{j,k}\), \(\mathcal M_{j,k}^{c}(S^{4n+3})\), and \(\Theta(j,k;\theta,4n+3)\) replaced by \(\overline{\mathcal P}_{\ell}\), \(\mathcal M_{\ell}^{c}(S^{4n+3})\), and \(\Theta(\ell;\theta,4n+3)\), respectively.
\end{proof}

\begin{Rek}\label{rem:first-moment-balance}
When \(\ell=1\), the condition
\[
\int_{S^{4n+3}} g\,d\nu=0
\qquad
\text{for all } g\in \overline{\mathcal P}_{1}
\]
is equivalent to the balance condition
\[
\int_{S^{4n+3}} \xi_i\,d\nu=0,
\qquad
i=1,\dots,4n+4,
\]
where \(\xi_1,\dots,\xi_{4n+4}\) are the coordinate functions on \(\mathbb{R}^{4n+4}\).
\end{Rek}

\begin{proof}[Proof of Corollary~\ref{cor:first-moment}]
This is exactly Theorem~\ref{thm:improved-l} with \(\ell=1\).
\end{proof}

\begin{Rek}
The proof of Lemma~\ref{lem:cc} only uses the positivity of the sharp constant \(S_p^{QC}\). No explicit formula for \(S_p^{QC}\) is needed.
\end{Rek}

\begin{Rek}
If \(j+k=\ell\), then \(\mathcal P_{j,k}\subset \mathcal P_{\ell}\), and therefore
\[
\overline{\mathcal P}_{j,k}\subset \overline{\mathcal P}_{\ell}.
\]
Consequently,
\[
\mathcal M_{\ell}^{c}(S^{4n+3})
\subset
\mathcal M_{j,k}^{c}(S^{4n+3}),
\]
which implies
\[
\Theta\!\left(\ell;\frac{Q-p}{Q},4n+3\right)
\ge
\Theta\!\left(j,k;\frac{Q-p}{Q},4n+3\right).
\]
Hence
\[
\frac{S_p^{QC}}{\Theta\!\left(\ell;\frac{Q-p}{Q},4n+3\right)}
\le
\frac{S_p^{QC}}{\Theta\!\left(j,k;\frac{Q-p}{Q},4n+3\right)}.
\]
\end{Rek}

\begin{Rek}
Some exact values of \(\Theta(\ell;\theta,N-1)\) are
\[
\Theta(1;\theta,N-1)=2^{1-\theta},
\qquad
\Theta(2;\theta,N-1)=(N+1)^{1-\theta},
\]
proved in \cite{MR4452442}, and
\[
\Theta(3;\theta,N-1)=(2N)^{1-\theta},
\]
proved in \cite{2020Cubature}. It is also shown in \cite{2020Cubature} that \(\Theta(\ell;\theta,N-1)\) is attained by a finitely supported probability measure. In particular, the defining infimum is in fact a minimum.
\end{Rek}

\section{Existence of the Extremals}

We now turn to the existence of extremals for the sharp Sobolev inequality \eqref{eq:sharp-sobolev}. We begin by writing down an explicit family of quaternionic contact conformal automorphisms on the sphere, analogous to \cite[(B.1)]{FL10}.

Let
\[
\delta_\lambda(q,\omega)=(\lambda q,\lambda^2\omega),
\qquad
\lambda>0,
\]
denote the Heisenberg dilation on the quaternionic Heisenberg group \(\mathbb H^n\times \Im\mathbb H\).

Write a point on the sphere \(S^{4n+3}\) as
\[
\zeta=(\zeta',\zeta_{n+1}),
\qquad
\zeta'\in \mathbb H^n,\ \zeta_{n+1}\in \mathbb H.
\]
The quaternionic Cayley transform
\[
\mathcal C:\mathbb H^n\times \Im\mathbb H\to S^{4n+3}\setminus\{(0,\dots,0,-1)\}
\]
is given by
\[
\mathcal C(q,\omega)
=
\left(
\frac{2q}{1+|q|^2+\omega},
\,
\frac{1-|q|^2-\omega}{1+|q|^2+\omega}
\right),
\]
and its inverse is
\[
\mathcal C^{-1}(\zeta)
=
\left(
\frac{\zeta'}{1+\zeta_{n+1}},
\,
\Im \frac{1-\zeta_{n+1}}{1+\zeta_{n+1}}
\right).
\]

Let \(N=(0,\dots,0,1)\in S^{4n+3}\). Define
\[
\Gamma_{\lambda,N}
=
\mathcal C\circ \delta_\lambda\circ \mathcal C^{-1}.
\]
A direct computation yields
\[
\Gamma_{\lambda,N}(\zeta',\zeta_{n+1})
=
\left(
\frac{2\lambda\,\zeta'}
{1+\zeta_{n+1}+\lambda^2(1-\zeta_{n+1})},
\,
\frac{1+\zeta_{n+1}-\lambda^2(1-\zeta_{n+1})}
{1+\zeta_{n+1}+\lambda^2(1-\zeta_{n+1})}
\right).
\]

More generally, for any \(\xi\in S^{4n+3}\), choose \(A_\xi\in Sp(n+1)\) such that
\[
A_\xi(\xi)=N.
\]
We then define
\[
\Gamma_{\lambda,\xi}
=
A_\xi^{-1}\circ \Gamma_{\lambda,N}\circ A_\xi
=
A_\xi^{-1}\circ \mathcal C\circ \delta_\lambda\circ \mathcal C^{-1}\circ A_\xi.
\]
This yields the required family of quaternionic contact conformal automorphisms of \(S^{4n+3}\).

\begin{Lem}\label{lem:qc-balance}
Let \(r>0\), and let \(u\in L^{r}(S^{4n+3})\) be nontrivial. Then there exists
\(\Phi\in \mathrm{Aut}_{QC}(S^{4n+3})\) such that
\[
u^\Phi(\zeta)=|J_\Phi(\zeta)|^{1/r}u(\Phi(\zeta))
\]
satisfies
\[
\int_{S^{4n+3}} \zeta\,|u^\Phi(\zeta)|^{r}\,d\xi=0
\]
in \(\mathbb{R}^{4n+4}\), where \(S^{4n+3}\subset \mathbb{R}^{4n+4}\).
\end{Lem}

\begin{proof}
Set
\[
f(\zeta)=\frac{|u(\zeta)|^{r}}{\int_{S^{4n+3}}|u|^{r}\,d\xi}.
\]
Then \(f\ge 0\), \(f\in L^{1}(S^{4n+3})\), and
\[
\int_{S^{4n+3}} f\,d\xi=1.
\]

Let \(\Gamma_{\lambda,\xi}\in \mathrm{Aut}_{QC}(S^{4n+3})\), with \(\lambda>0\) and \(\xi\in S^{4n+3}\), denote the standard family of quaternionic contact conformal automorphisms obtained by conjugating Heisenberg dilations with the quaternionic Cayley transform. We only use the following properties:

\begin{itemize}
\item[(i)] \(\Gamma_{1,\xi}=\mathrm{id}_{S^{4n+3}}\) for every \(\xi\in S^{4n+3}\);
\item[(ii)] for every fixed \(\xi\in S^{4n+3}\),
\[
\Gamma_{\lambda,\xi}(\zeta)\to \xi
\qquad
\text{as }\lambda\downarrow 0
\]
for every \(\zeta\neq -\xi\), and the convergence is uniform on compact subsets of \(S^{4n+3}\setminus\{-\xi\}\);
\item[(iii)] the map \((\lambda,\xi,\zeta)\mapsto \Gamma_{\lambda,\xi}(\zeta)\) is continuous on its domain.
\end{itemize}

Let
\[
B=\{x\in \mathbb R^{4n+4}:|x|<1\},
\qquad
\overline B=\{x\in \mathbb R^{4n+4}:|x|\le 1\}.
\]
Define \(F:B\to \mathbb R^{4n+4}\) by
\[
F(r\xi)=\int_{S^{4n+3}}\Gamma_{1-r,\xi}(\zeta)\,f(\zeta)\,d\xi(\zeta),
\qquad
0<r<1,\ \xi\in S^{4n+3}.
\]
Since \(\Gamma_{1-r,\xi}(\zeta)\in S^{4n+3}\) for every \(\zeta\), \(F(r\xi)\) is a convex combination of points on the sphere. Hence
\[
F(B)\subset \overline B.
\]

Since \(\Gamma_{1,\xi}=\mathrm{id}_{S^{4n+3}}\), we have
\[
\lim_{r\downarrow 0}F(r\xi)
=
\int_{S^{4n+3}}\zeta\,f(\zeta)\,d\xi(\zeta),
\]
which is independent of \(\xi\). Therefore \(F\) extends continuously to the origin by setting
\[
F(0):=\int_{S^{4n+3}}\zeta\,f(\zeta)\,d\xi(\zeta).
\]

We next examine the behavior of \(F\) near \(\partial B\). Fix \(\varepsilon>0\). Since \(f\in L^{1}(S^{4n+3})\), there exists \(\rho>0\) such that
\[
\int_E f\,d\xi<\varepsilon
\]
whenever \(E\subset S^{4n+3}\) is measurable and \(|E|<\rho\). For \(\eta>0\) and \(\xi\in S^{4n+3}\), let
\[
E_{\eta,\xi}
=
\{\zeta\in S^{4n+3}:|1+\langle \zeta,\xi\rangle|<\eta\}.
\]
By rotational invariance of the surface measure, \(|E_{\eta,\xi}|\) is independent of \(\xi\). Thus, choosing \(\eta>0\) sufficiently small, we may arrange that \(|E_{\eta,\xi}|<\rho\) for every \(\xi\in S^{4n+3}\), and hence
\[
\int_{E_{\eta,\xi}} f\,d\xi<\varepsilon
\qquad
\text{for every }\xi\in S^{4n+3}.
\]

Now
\[
\begin{aligned}
|F(r\xi)-\xi|
&\le
\int_{E_{\eta,\xi}}
|\Gamma_{1-r,\xi}(\zeta)-\xi|\,f(\zeta)\,d\xi(\zeta)
+
\int_{S^{4n+3}\setminus E_{\eta,\xi}}
|\Gamma_{1-r,\xi}(\zeta)-\xi|\,f(\zeta)\,d\xi(\zeta) \\
&\le
2\int_{E_{\eta,\xi}} f\,d\xi
+
\sup_{\zeta\in S^{4n+3}\setminus E_{\eta,\xi}}
|\Gamma_{1-r,\xi}(\zeta)-\xi|.
\end{aligned}
\]
The first term is bounded by \(2\varepsilon\). For the second term, note that
\[
\{(\xi,\zeta)\in S^{4n+3}\times S^{4n+3}:|1+\langle \zeta,\xi\rangle|\ge \eta\}
\]
is compact and avoids the pole \(\zeta=-\xi\). Hence, by property (ii) together with continuity, we have
\[
\sup_{\substack{\xi\in S^{4n+3}\\ \zeta\in S^{4n+3}\setminus E_{\eta,\xi}}}
|\Gamma_{1-r,\xi}(\zeta)-\xi|
\to 0
\qquad
\text{as }r\uparrow 1.
\]
Therefore
\[
\lim_{r\uparrow 1}F(r\xi)=\xi
\]
uniformly in \(\xi\in S^{4n+3}\). It follows that \(F\) extends continuously to \(\overline B\) and satisfies
\[
F(x)=x
\qquad
\text{for every }x\in \partial B=S^{4n+3}.
\]

We claim that \(0\in F(\overline B)\). Suppose not. Then
\[
R(x)=\frac{F(x)}{|F(x)|}
\]
defines a continuous map \(R:\overline B\to \partial B\). Since \(F(x)=x\) for \(x\in \partial B\), we also have \(R(x)=x\) on \(\partial B\). Thus \(R\) is a retraction of \(\overline B\) onto \(\partial B\), which is impossible. Hence there exists \(x_0=r_0\xi_0\in \overline B\) such that
\[
F(x_0)=0.
\]

Let
\[
\Gamma=\Gamma_{1-r_0,\xi_0},
\qquad
\Phi=\Gamma^{-1}.
\]
Then
\[
\int_{S^{4n+3}}\Gamma(\zeta)\,f(\zeta)\,d\xi(\zeta)=0.
\]

Finally, using the change of variables \(\eta=\Phi(\zeta)\), we obtain
\[
\begin{aligned}
\int_{S^{4n+3}}\zeta\,|u^\Phi(\zeta)|^r\,d\xi(\zeta)
&=
\int_{S^{4n+3}}\zeta\,|J_\Phi(\zeta)|\,|u(\Phi(\zeta))|^r\,d\xi(\zeta) \\
&=
\int_{S^{4n+3}}\Phi^{-1}(\eta)\,|u(\eta)|^r\,d\xi(\eta) \\
&=
\int_{S^{4n+3}}\Gamma(\eta)\,|u(\eta)|^r\,d\xi(\eta) \\
&=
\left(\int_{S^{4n+3}}|u|^r\,d\xi\right)
\int_{S^{4n+3}}\Gamma(\eta)\,f(\eta)\,d\xi(\eta) \\
&=0.
\end{aligned}
\]
This proves the lemma.
\end{proof}

\begin{proof}[Proof of Theorem~\ref{thm:extremal}]
Let \(\{u_m\}\subset S^{1,2}(S^{4n+3})\) be a minimizing sequence for \eqref{eq:sharp-sobolev}. After normalization, we may assume that
\[
\|u_m\|_{L^{2^*}(S^{4n+3})}=1
\qquad
\text{for every } m,
\]
and
\[
\int_{S^{4n+3}} |\nabla_{qc}u_m|^2\,d\xi
\longrightarrow
\frac{1}{S_2^{QC}}.
\]

By Lemma~\ref{lem:qc-balance} with \(r=2^*\), for each \(m\) there exists
\(\Phi_m\in \mathrm{Aut}_{QC}(S^{4n+3})\) such that
\[
v_m
=
u_m^{\Phi_m}
=
|J_{\Phi_m}|^{1/2^*}\,\Phi_m^*u_m
\]
satisfies the balance condition
\[
\int_{S^{4n+3}} \zeta\,|v_m(\zeta)|^{2^*}\,d\xi=0.
\]
Since the sharp Sobolev quotient is invariant under quaternionic contact conformal automorphisms, \(\{v_m\}\) is still a minimizing sequence. In particular,
\[
\|v_m\|_{L^{2^*}(S^{4n+3})}=1
\qquad
\text{for every } m,
\]
and
\[
\int_{S^{4n+3}} |\nabla_{qc}v_m|^2\,d\xi
\longrightarrow
\frac{1}{S_2^{QC}}.
\]

Passing to a subsequence, still denoted by \(\{v_m\}\), we may assume that
\[
v_m \rightharpoonup u
\qquad
\text{weakly in } S^{1,2}(S^{4n+3}).
\]
Since \(S^{4n+3}\) is compact, the embedding
\[
S^{1,2}(S^{4n+3}) \hookrightarrow L^2(S^{4n+3})
\]
is compact. Therefore
\[
v_m\to u
\qquad
\text{strongly in } L^2(S^{4n+3}).
\]
After passing to a further subsequence if necessary, we may also assume that
\[
v_m(\zeta)\to u(\zeta)
\qquad
\text{for a.e. }\zeta\in S^{4n+3}.
\]

We first show that \(u\neq 0\). By Corollary~\ref{cor:first-moment} with \(p=2\),
\[
\|v_m\|_{L^{2^*}(S^{4n+3})}^{2}
\le
\left(
\frac{S_2^{QC}}{\Theta\!\left(1;\frac{Q-2}{Q},4n+3\right)}
+\varepsilon
\right)
\int_{S^{4n+3}} |\nabla_{qc}v_m|^2\,d\xi
+
C(\varepsilon)\int_{S^{4n+3}} |v_m|^2\,d\xi.
\]
Since
\[
\Theta\!\left(1;\frac{Q-2}{Q},4n+3\right)=2^{2/Q}>1,
\]
we may choose \(\varepsilon>0\) so small that
\[
A_\varepsilon
:=
\frac{S_2^{QC}}{\Theta\!\left(1;\frac{Q-2}{Q},4n+3\right)}
+\varepsilon
<
S_2^{QC}.
\]
Hence
\[
1
\le
A_\varepsilon \int_{S^{4n+3}} |\nabla_{qc}v_m|^2\,d\xi
+
C(\varepsilon)\int_{S^{4n+3}} |v_m|^2\,d\xi.
\]
Letting \(m\to\infty\), and using the strong convergence in \(L^2(S^{4n+3})\), we obtain
\[
1
\le
\frac{A_\varepsilon}{S_2^{QC}}
+
C(\varepsilon)\int_{S^{4n+3}} |u|^2\,d\xi.
\]
Since \(A_\varepsilon<S_2^{QC}\), it follows that
\[
\int_{S^{4n+3}} |u|^2\,d\xi
\ge
\frac{1-A_\varepsilon/S_2^{QC}}{C(\varepsilon)}
>0.
\]
Therefore \(u\neq 0\).

We next show that \(u\) is an extremal. By the Brézis--Lieb lemma,
\[
1
=
\|v_m\|_{L^{2^*}(S^{4n+3})}^{2^*}
=
\|u\|_{L^{2^*}(S^{4n+3})}^{2^*}
+
\|v_m-u\|_{L^{2^*}(S^{4n+3})}^{2^*}
+
o(1).
\]
Set
\[
a=\|u\|_{L^{2^*}(S^{4n+3})}^{2^*}\in (0,1].
\]
Then
\[
\|v_m-u\|_{L^{2^*}(S^{4n+3})}^{2^*}\to 1-a.
\]

Moreover, since \(v_m\rightharpoonup u\) weakly in \(S^{1,2}(S^{4n+3})\), we have the energy splitting
\[
\int_{S^{4n+3}} |\nabla_{qc}v_m|^2\,d\xi
=
\int_{S^{4n+3}} |\nabla_{qc}u|^2\,d\xi
+
\int_{S^{4n+3}} |\nabla_{qc}(v_m-u)|^2\,d\xi
+
o(1).
\]
Applying \eqref{eq:sharp-sobolev} to \(u\) and to \(v_m-u\), we obtain
\[
\|u\|_{L^{2^*}(S^{4n+3})}^{2}
\le
S_2^{QC}\int_{S^{4n+3}} |\nabla_{qc}u|^2\,d\xi,
\]
and
\[
\|v_m-u\|_{L^{2^*}(S^{4n+3})}^{2}
\le
S_2^{QC}\int_{S^{4n+3}} |\nabla_{qc}(v_m-u)|^2\,d\xi.
\]
Therefore,
\[
\begin{aligned}
1
&=
\lim_{m\to\infty}
S_2^{QC}\int_{S^{4n+3}} |\nabla_{qc}v_m|^2\,d\xi \\
&=
S_2^{QC}\int_{S^{4n+3}} |\nabla_{qc}u|^2\,d\xi
+
\lim_{m\to\infty}
S_2^{QC}\int_{S^{4n+3}} |\nabla_{qc}(v_m-u)|^2\,d\xi \\
&\ge
\|u\|_{L^{2^*}(S^{4n+3})}^{2}
+
\lim_{m\to\infty}\|v_m-u\|_{L^{2^*}(S^{4n+3})}^{2} \\
&=
a^{2/2^*}+(1-a)^{2/2^*}.
\end{aligned}
\]
Since the function \(t\mapsto t^{2/2^*}\) is strictly concave on \([0,1]\), one has
\[
a^{2/2^*}+(1-a)^{2/2^*}>1
\qquad
\text{for every } a\in(0,1).
\]
Thus the above inequality forces \(a=1\). Hence
\[
\|u\|_{L^{2^*}(S^{4n+3})}=1
\qquad
\text{and}
\qquad
\|v_m-u\|_{L^{2^*}(S^{4n+3})}\to 0.
\]

Finally,
\[
1
=
\|u\|_{L^{2^*}(S^{4n+3})}^{2}
\le
S_2^{QC}\int_{S^{4n+3}} |\nabla_{qc}u|^2\,d\xi
\le
S_2^{QC}\liminf_{m\to\infty}\int_{S^{4n+3}} |\nabla_{qc}v_m|^2\,d\xi
=1.
\]
Therefore equality holds throughout, and \(u\) is an extremal for \eqref{eq:sharp-sobolev}. In particular,
\[
\int_{S^{4n+3}} |\nabla_{qc}u|^2\,d\xi=\frac{1}{S_2^{QC}}.
\]
Combining this with the energy splitting, we conclude that
\[
\int_{S^{4n+3}} |\nabla_{qc}(v_m-u)|^2\,d\xi\to 0.
\]
Together with the strong convergence \(v_m\to u\) in \(L^2(S^{4n+3})\), this implies
\[
v_m\to u
\qquad
\text{strongly in } S^{1,2}(S^{4n+3}).
\]

Recalling that \(v_m=u_m^{\Phi_m}\), the proof is complete.
\end{proof}

\section*{Acknowledgments}

\medskip
{\bf Funding:} This work is supported by National Natural Science Foundation of China (12301145, 12261107, 12561020) and Yunnan Fundamental Research Projects (202301AU070144, 202401AU070123).

\medskip
{\bf Author Contributions:} All the authors wrote the main manuscript text together and these authors contributed equally to this work.

\medskip
{\bf Data availability:}  Data sharing is not applicable to this article as no new data were created or analyzed in this study.

\medskip
{\bf Conflict of Interests:} The authors declare that there is no conflict of interest.

\bibliographystyle{plain} 
\bibliography{reference} 
\end{document}